\documentclass{icmart}

\usepackage{amscd}

\contact[mmustata@umich.edu]{Department of Mathematics, University of Michigan, Ann Arbor, MI 48109, USA}

\newtheorem{theorem}{Theorem}[section]
\newtheorem{corollary}[theorem]{Corollary}

\newtheorem{proposition}[theorem]{Proposition}

\theoremstyle{definition}

\newtheorem{remark}[theorem]{Remark}
\newtheorem{example}[theorem]{Example}
\newtheorem{question}[theorem]{Question}

\newcommand{\Hom}{\operatorname{Hom}}

\newcommand{\Spec}{\operatorname{Spec}}

\def\ZZ{{\mathbf Z}}

\def\AA{{\mathbf A}}

\def\cO{\mathcal{O}}

\def\fra{\mathfrak{a}}

\def\frm{\mathfrak{m}}

\newcommand{\ord}{\textnormal{ord}}

\newcommand{\lct}{\textnormal{lct}}

\newcommand{\mld}{\textnormal{mld}}

\newcommand{\codim}{\textnormal{codim}}

\newcommand{\ilim}{\mathop{\varinjlim}\limits}

    \newcommand{\plim}{\mathop{\varprojlim}\limits}

\newcommand{\llbracket}{[\negthinspace[}
\newcommand{\rrbracket}{]\negthinspace]}

\title[The dimension of jet schemes of singular varieties]{The dimension of jet schemes of singular varieties}

\author[Mircea Musta\c{t}\u{a}]
{Mircea Musta\c{t}\u{a}\thanks{The author was partially supported by NSF grants
DMS-1068190 and DMS-1265256.}}

\begin{document}

\begin{abstract}
Given a scheme $X$ over $k$, a generalized jet scheme parametrizes maps $\Spec A\to X$, where
$A$ is a finite-dimensional, local algebra over $k$. We give an overview of known results concerning the 
dimensions of these schemes for $A=k[t]/(t^m)$, when they are related to 
invariants of singularities in birational geometry. We end with a discussion of more general jet schemes.

\end{abstract}



\maketitle

\section{Introduction}

Given a scheme $X$ (say, of finite type over an algebraically closed field $k$), a tangent vector to $X$ can be identified with
 a morphism $\Spec k[t]/(t^2) \to X$. The tangent vectors to $X$ are the closed points of the first jet scheme $J_1(X)$ of $X$. More generally,
for every $m\geq 1$ one can define a scheme of finite type $J_m(X)$ whose closed points parametrize all morphisms
$\Spec k[t]/(t^{m+1}) \to X$. Explicitly, if $X$ is defined in some affine space $\AA^N$ by polynomials $f_1,\ldots,f_r$,
then as a set $J_m(X)$ is identified to the set of solutions of $f_1,\ldots,f_r$ in $k[t]/(t^{m+1})$. 
Truncating induces natural maps $J_p(X)\to J_q(X)$ for every $p>q$. When $X$ is a smooth, $n$-dimensional variety, then every such projection is
locally trivial in the Zariski topology, with fiber $\AA^{(p-q)n}$. In particular, $J_m(X)$ is a smooth variety of dimension $(m+1)n$. However, when $X$ is singular,
the jet schemes $J_m(X)$ have a more complicated behavior that reflects in a subtle way the singularities of $X$. 

In fact, instead of the algebras $k[t]/(t^{m+1})$ we can consider an arbitrary local, finite $k$-algebra $A$. The morphisms
$\Spec(A)\to X$ are parametrized by a generalized jet scheme that we denote $J_A(X)$. 
For example, if 
$$A=k[t_1,\ldots,t_r]/(t_1^{m_1+1},\ldots,t_r^{m_r+1}),$$ then $J_A(X)$ is isomorphic to the $r$-iterated jet scheme
$J_{m_1}(J_{m_2}(\ldots J_{m_r}(X)))$. 
The construction and the formal properties
of these more general jet schemes are very similar to those of the usual ones. We give an overview of the construction and of the basic properties
of the generalized jet schemes in \S 2. By taking suitable projective limits, one can then define $J_A(X)$ when $A$ is a local, complete $k$-algebra, with residue field $k$.
We describe this construction in \S 3. The much-studied case is that of $A=k\llbracket t\rrbracket$, when $J_A(X)$ is known as the scheme of arcs of $X$. 

Information on the schemes $J_m(X)$  is provided by the change of variable formula in motivic integration, see \cite{DL99}. 
More precisely, if $X$ is embedded in a smooth variety $Y$, then one can  use a log resolution of the pair $(Y,X)$ to compute, for example, the dimensions of the schemes $J_m(X)$
in terms of the data of the resolution. In this way one can relate the behavior of these dimensions to some of the invariants of the pair $(Y,X)$ that appear in birational geometry.
We describe this story in \S 4. 

One can ask questions with a similar flavor about the dimensions of the generalized jet schemes $J_A(X)$, but very little is known in this direction. We propose in \S 5
some invariants defined in terms of the asymptotic behavior of the dimensions of $J_A(X)$, when $A$ varies over certain sequences of algebras of embedding dimension $2$.
We also discuss the irreducibility of iterated jet schemes for locally complete intersection varieties.

\section{Generalized jet schemes}

Let $k$ be a field of arbitrary characteristic. 
All schemes we consider are schemes over $k$.
Let ${\rm LFA}/k$ be the category whose objects are local finite $k$-algebras, with residue field $k$,
with the maps being local homomorphisms of $k$-algebras.
Given a scheme $X$ 
 and $A\in {\rm LFA}/k$, 
the scheme of $A$-jets of $X$ is a scheme $J_A(X)$ that satisfies the following universal property: for every scheme $Y$, there is a functorial bijection
\begin{equation}\label{eq_def1}
\Hom_{{\rm Sch}/k}(Y,J_A(X))\simeq \Hom_{{\rm Sch}/k}(Y\times\Spec A,X).
\end{equation}
Standard arguments imply that it is enough to have a functorial bijection
as in (\ref{eq_def1}) when $Y$ is an affine scheme. In particular, by taking $Y=\Spec k$, we see that
the set of $k$-valued points of $J_A(X)$ is in bijection with the set of $A$-\emph{jets} of $X$, that is, $A$-valued points of $X$. 
Note that if $A=k$, then $J_A(X)=X$.

Before discussing existence, we make some general remarks. Suppose that  $\phi\colon A\to B$ is a homomorphism in ${\rm LFA}/k$ and that $J_A(X)$ and $J_B(X)$ exist. We have a functorial transformation 
$$ \Hom_{{\rm Sch}/k}(Y\times\Spec A,X)\to \Hom_{{\rm Sch}/k}(Y\times\Spec B,X)$$
given by composition with $Y\times \Spec B\to Y\times\Spec A$. This is
induced via the isomorphism (\ref{eq_def1}) by a unique scheme morphism 
$\pi^X_{B/A}\colon J_A(X)\to J_B(X)$. In particular, if $A\in {\rm LFA}/k$, then the 
morphism to the residue field $A\to k$ 
induces
$\pi^X_A=\pi^X_{A/k}\colon J_A(X)\to X$. Similarly, the structure morphism
$k\to A$ induces a section $s^X_A\colon X\to J_A(X)$ of $\pi_A^X$.
Given two morphisms  $A\to B$ and $B\to C$ in ${\rm LFA}/k$, if
$J_A(X)$, $J_B(X)$, and $J_C(X)$ exist, then it is easy to see that $\pi^X_{C/B}\circ \pi^X_{B/A}=\pi^X_{C/A}$. 
When the scheme $X$ is clear from the context, we simply write $\pi_{B/A}$,
$\pi_A$, and $s_A$ instead of $\pi^X_{B/A}$, $\pi^X_A$, and $s_A^X$.

When $A=k[t]/(t^{m+1})$, the $A$-jets are called $m$-jets and the corresponding scheme is denoted $J_m(X)$, the
$m^{\rm th}$ \emph{jet scheme} of $X$. It is not hard to deduce from the universal property that $J_1(X)$
is isomorphic as a scheme over $X$ with the total tangent space ${\mathcal Spec}({\rm Sym}(\Omega_{X/k}))$. 
We now sketch the proof of the existence of $J_A(X)$. The argument is the same as in
the case of the usual jet schemes, hence we refer the reader to \cite[Section 2]{EM09} for details. 

Consider first the case when $X=\Spec S$ is affine and consider a presentation 
$$S\simeq k[x_1,\ldots,x_N]/(f_1,\ldots,f_r).$$
Let us fix a basis $(e_i)_{1\leq i\leq m}$ for $A$. We can thus write 
\begin{equation}\label{eq2_adjoint_jets}
e_i\cdot e_j=\sum_{\ell=1}^mc_{i,j,\ell}e_{\ell}.
\end{equation}
We consider an affine scheme $Y=\Spec R$. 
Giving a morphism $Y\times\Spec A \to X$ is equivalent to giving a ring homomorphism
\begin{equation}\label{eq3_adjoint_jets}
\phi\colon k[x_1,\ldots,x_N]/(f_1,\ldots,f_r)\to R\otimes_kA.
\end{equation}
This is uniquely determined by $\phi(x_j)$, for $1\leq j\leq N$, which can be written as
$$\phi(x_j)=\sum_{i=1}^ma_{i,j}\otimes e_i,\,\,\text{with}\,\,a_{i,j}\in R\,\,\text{for}\,\,1\leq i\leq m,\,1\leq j\leq N.$$
Furthermore, for any such choice of $\phi(x_j)$, we get a $k$-algebra homomorphism
$\widetilde{\phi}\colon k[x_1,\ldots,x_N]\to R\otimes_kA$ and this induces a $k$-algebra homomorphism 
(\ref{eq3_adjoint_jets}) if and only if $\widetilde{\phi}(f_{\alpha})=0$ for $1\leq\alpha\leq r$. 
On the other hand, the relations (\ref{eq2_adjoint_jets}) imply that for every $\alpha$, we can find
polynomials $P_{\alpha}^{(i)}$ in $a_{j,\ell}$, with coefficients in $k$ (these coefficients in turn are polynomials in the structure constants
$c_{i,j,\ell}$) such that for $\widetilde{\phi}$ as above, we have
$$\widetilde{\phi}(f_{\alpha})=\sum_{i=1}^m P_{\alpha}^{(i)}(a_{1,1},\ldots,a_{m,N})\otimes e_i.$$
This shows that $J_{A}(X)$ is cut out in $\AA^{mN}$ by the equations $P_{\alpha}^{(i)}$ for $1\leq\alpha\leq r$ and $1\leq i\leq m$.

The above argument shows that $J_A(X)$ exists whenever $X$ is affine. It is then easy to check that if
$X$ is any scheme such that $J_A(X)$ exists, then for every open subset $U$ of $X$, the scheme $J_A(U)$ exists
and it is isomorphic to $(\pi^X_A)^{-1}(U)$. Given now an arbitrary scheme $X$, consider an affine open cover $X=\cup_iU_i$.
Note that $J_A(U_i)$ exists for every $i$. Moreover, for every $i$ and $j$, the schemes $(\pi_A^{U_i})^{-1}(U_i\cap U_j)$ and 
$(\pi_A^{U_j})^{-1}(U_i\cap U_j)$ are canonically isomorphic, being isomorphic to $J_A(U_i\cap U_j)$. We can thus glue
the schemes $J_A(U_i)$ along these open subsets and it is then straightforward to check that the resulting scheme satisfies the
universal property of $J_A(X)$. We collect in the next proposition the conclusion of the above discussion.

\begin{proposition}\label{prop_existence_jet_scheme}
For every $A\in {\rm LFA}/k$ and every scheme $X$, the scheme $J_A(X)$ of $A$-jets of $X$ exists. Moreover, the following hold:
\begin{enumerate}
\item[i)] If $X$ is of finite type over $k$, then $J_A(X)$ is of finite type over $k$.
\item[ii)] The canonical projection $\pi_A\colon J_A(X)\to X$ is affine. 
\item[iii)] If $X$ is an affine subscheme of ${\mathbf A}^N$ defined by $r$ equations and $\dim_k(A)=m$, then
$J_A(X)$ is defined in $J_A({\mathbf A}^N)\simeq {\mathbf A}^{Nm}$ by $rm$ equations. 
More generally, if $A\to A'$ is a surjective homomorphism in ${\rm LFA}/k$, then
$$J_A(X)\hookrightarrow (\pi^{\AA^N}_{A'/A})^{-1}(J_{A'}(X))$$
is cut out by $r\cdot (\dim_k(A)-\dim_k(A'))$ equations.
\end{enumerate}
\end{proposition}

In what follows we discuss some basic properties of generalized jet schemes.
It is clear that if $f\colon X\to Z$ is a morphism and $A\in {\rm LFA}/k$, we have a unique morphism
$J_A(f)\colon J_A(X)\to J_A(Z)$ such that the bijection (\ref{eq_def1}) is functorial also in
 $X$. Therefore taking $X$ to $J_A(X)$ gives a functor, such that if $A\to B$ is a morphism of finite 
 $k$-algebras, then $\pi_{B/A}$ is a natural transformation.
 
 \begin{example}\label{ex_iterated}
 Iterated jet schemes can be described as schemes of $A$-jets. Indeed, given 
 $m_1,\ldots,m_r\in\ZZ_{\geq 0}$, if we take $A=k[t_1,\ldots,t_r]/(t_1^{m_1+1},\ldots,t_r^{m_r+1})$, then it follows from the universal property that for every scheme $X$, we have a canonical isomorphism of schemes over 
 $X$
 $$J_A(X)\simeq J_{m_1}(J_{m_2}(\ldots J_{m_r}(X))).$$
 \end{example}
 
 \begin{remark}
 It follows from the explicit description of the scheme of $A$-jets of an affine scheme $X$ that if $\iota\colon Z\hookrightarrow X$ is a closed immersion,
 then for every $A\in {\rm LFA}/k$, the induced morphism $J_A(\iota)$ is a closed immersion.
 \end{remark}
 
\begin{remark}\label{field_extension}
If $X$ is any scheme over $k$ and $K/k$ is a field extension, for every 
$A\in {\rm LFA}/k$ we have $A\otimes_kK\in {\rm LFA}/K$ and there is a canonical isomorphism 
$$J_{A\otimes_kK}(X\times_{\Spec k}\Spec K)\simeq J_A(X)\times_{\Spec k}\Spec K.$$
The assertion follows easily from the isomorphism (\ref{eq_def1}).
\end{remark}

The following proposition describing the behavior of generalized jet schemes with respect to \'{e}tale morphisms is an 
immediate consequence of (\ref{eq_def1}) and of the fact that \'{e}tale morphisms are formally \'{e}tale.

\begin{proposition}\label{general_properties3}
If $f\colon X\to Y$ is an \'{e}tale morphism of schemes of finite type over $k$, then for every $A\in {\rm LFA}/k$ we have a Cartezian diagram
\[
\begin{CD}
J_A(X) @>{J_A(f)}>> J_A(Y) \\
@V{\pi_A^X}VV@VV{\pi_A^Y}V \\
X@>{f}>>  Y.
\end{CD}
\]
\end{proposition}

Using Proposition~\ref{general_properties3} and the description of the scheme of $A$-jets
for an affine space, we obtain the following

\begin{corollary}\label{smooth_case}
If $X$ is a smooth $n$-dimensional variety\footnote{A \emph{variety} is assumed to be reduced, irreducible, and of finite type over $k$.} over $k$, then for every surjective
morphism $A\to B$ in ${\rm LFA}/k$, the induced morphism 
$J_A(X)\to J_B(X)$ is locally trivial in the Zariski topology, with fiber
$\AA^{dn}$, where $d=\dim_k(A)-\dim_k(B)$. In particular, $J_A(X)$ is a smooth variety
of dimension $n\cdot \dim_k(A)$. 
\end{corollary}

\begin{corollary}\label{bound_dimension_jet_schemes}
If $X$ is a scheme of finite type over $k$, then for every
$A\in {\rm LFA}/k$, we have 
$$\dim(X)\leq\frac{\dim(J_A(X))}{\dim_k(A)}\leq \max_{x\in X}\dim(T_xX).$$
\end{corollary}

\begin{proof}
The first inequality follows from Corollary~\ref{smooth_case} and the fact that there is a locally closed immersion $Y\hookrightarrow X$,
with $Y$ smooth and $\dim(Y)=\dim(X)$. The second inequality also follows from Corollary~\ref{smooth_case} since 
for every $x\in X$, if $n=\dim(T_xX)$, then there is an open neighborhood $U$ of $x$ in $X$ and a closed immersion
$U\hookrightarrow Z$, where $Z$ is a smooth variety of dimension $n$.
\end{proof}

\begin{example}\label{example1}
If we consider arbitrary $A\in {\rm LFA}/k$, then we can get arbitrarily close to the upper-bound in Corollary~\ref{bound_dimension_jet_schemes}.
Indeed, if $x\in X$ is such that $n=\dim(T_xX)$ and $(A,{\mathfrak m})\in {\rm LFA}/k$ is such that ${\mathfrak m}^2=0$ and $\dim_k(A)=r$,
then 
$$(\pi_A)^{-1}(x)\simeq \AA^{(r-1)n}\,\,\text{hence}\,\,\frac{\dim(J_A(X))}{\dim_k(A)}\geq n-\frac{n}{r}.$$
\end{example}


\begin{remark}
If $G$ is a group scheme over $k$, then it is easy to see that $J_A(G)$ is a group scheme over $k$ for every 
$A\in {\rm LFA}/k$. Moreover,
if $G$ acts on a scheme $X$, then $J_A(G)$ acts on $J_A(X)$.
\end{remark}

\begin{remark}\label{action}
Suppose that $A\in {\rm LFA}/k$ is a graded\footnote{All graded algebras we consider are graded by ${\mathbf Z}_{\geq 0}$.} $k$-algebra. The grading of $A$
induces a morphism $\phi\colon \AA^1\times\Spec(A)\to\Spec(A)$ that corresponds to the
$k$-algebra homomorphism $A\to A[t]$
that takes a homogeneous element 
$a\in A$ of degree $m$ to $at^m$. If $X$ is any scheme, then 
we obtain an induced morphism $\Phi_A^X\colon \AA^1\times J_A(X)\to J_A(X)$ 
that takes any morphism $(u,v)\colon Y\to\AA^1\times J_A(X)$,
with 
$$v\in {\rm Hom}(Y,J_A(X))\simeq {\rm Hom}(Y\times\Spec(A) ,X)$$ to
the composition
\[
\begin{CD}
Y\times\Spec(A)@>{({\rm id},u\circ {\rm pr}_1)}>>
Y\times\Spec(A)\times\AA^1@>{({\rm id}_Y,\phi})>> Y\times\Spec(A)@>{v}>>X.
\end{CD}
\]
It is easy to see that the restriction of $\Phi_A^X$ to $J_A(X)\simeq\{0\}\times J_A(X)$
is equal to $s_A^X\circ\pi_A^X$.
In particular, for such $A$ we see that if $Z$ is an irreducible component of $J_A(X)$, then 
$\Phi_A^X$ induces a morphism $\AA^1\times Z\to Z$, hence $s_A^X\circ\pi_A^X(Z)\subseteq
Z$.
\end{remark}

\section{Generalized arc schemes}

We now generalize the construction in the previous section to complete local algebras. More precisely,
let ${\rm LCA}/k$ be the category of complete local Noetherian $k$-algebras, with residue field $k$ (the maps being local morphisms of $k$-algebras). 
Note that ${\rm LFA}/k$ is a full subcategory of ${\rm LCA}/k$.

Given $(A,\frm)\in
{\rm LCA}/k$ and a scheme $X$ over $k$, we define a scheme $J_A(X)$
\begin{equation}\label{eq_def_gen_arcs}
J_A(X):=\plim_{A\to B}J_B(X),
\end{equation}
where the projective limit is over all surjective maps $A\to B$ in ${\rm LCA}/k$, with $B$ lying in ${\rm LFA}/k$. 
Note that a map from $\phi\colon A\to B$ to $\psi\colon A\to C$ is a map $f\colon B\to C$ such that
$f\circ \phi=\psi$ and such a map induces a morphism $\pi^X_{C/B}\colon J_B(X)\to J_C(X)$. 
Since all $J_B(X)$ are affine schemes over $X$, it follows that the projective limit (\ref{eq_def_gen_arcs}) exists.
In fact, if $U\subseteq X$ is an affine open subset, then
$$\Gamma(J_A(U),\cO_{J_A(U)})\simeq\ilim_{A\to B}\Gamma(J_B(U),\cO_{J_B(U)}).$$

Note that if $A\in {\rm LFA}/k$, then we recover the previous definition.
It is clear that if $h\colon X\to Y$ is a morphism of schemes, we obtain an induced morphism
$J_A(h)\colon J_A(X)\to J_A(Y)$ and in this way we get a functor from the category of schemes over $k$ to itself. 
If $g\colon A_1\to A_2$ is a morphism in ${\rm LCA}/k$, we obtain a functorial transformation
$\pi_{A_2/A_1}^X\colon J_{A_1}(X)\to J_{A_2}(X)$. Indeed, if $\phi\colon A_2\to B_2$ is a surjective map in ${\rm LCA}/k$,
with $B_2$ finite over $k$, then $\phi\circ g$ factors through a quotient $B_1$ of $A_1$ by a power of the maximal ideal,
hence we have a map $J_{A_1}(X)\to J_{B_1}(X)\to J_{B_2}(X)$, where the first map is given by the definition of projective limit
and the second map is $\pi^X_{B_2/B_1}$. Note that this definition has the following two properties:
\begin{enumerate}
\item[i)] If $\phi\colon A\to B$ is a surjective map in ${\rm LCA}/k$, with $B$ finite over $k$, then
the map $J_A(X)\to J_B(X)$ given by the projective limit definition is the same as $\pi_{B/A}^X$.
\item[ii)] If $A\to B\to C$ are maps in ${\rm LCA}/k$, then 
$\pi^X_{C/B}\circ\pi^X_{B/A}=\pi_{C/A}^X$.
\end{enumerate}

\begin{remark}\label{universal_property_arcs}
Suppose that $X$ is a scheme and $A\in {\rm LCA}/k$.
For every $k$-algebra $R$, consider the completion $R\widehat{\otimes}_kA$ of $R\otimes_kA$ with respect to the topology
induced by $A$ (more precisely, if ${\mathfrak m}_A$ is the maximal ideal in $A$, then the topology on $R\otimes_kA$
is the $\frm_A\cdot (R\otimes_kA)$-adic topology). In this case we have a canonical functorial map
\begin{equation}\label{eq_universal_property_arcs}
\Hom(\Spec R\widehat{\otimes}_kA,X)\to \Hom(\Spec R,J_A(X)).
\end{equation}
It is easy to see that this is a bijection if $R=k$ or if $X$ is affine.
\end{remark}

As we see in the next example, even when $X$ is a finite type, the scheme $J_A(X)$ is not, in general, of finite type.

\begin{example}
If $A=k\llbracket t\rrbracket$, the scheme $J_A(X)=\plim_mJ_m(X)$ is the \emph{scheme of arcs} of $X$, denoted by $J_{\infty}(X)$.
For example, if $X=\AA^n$, with $n\geq 1$, then $J_{\infty}(X)$ is an infinite-dimensional affine space, that is,
$J_{\infty}(X)\simeq\Spec k[x_1,x_2,\ldots]$. 

If $A=\llbracket t_1,t_2\rrbracket$, then $J_A(X)$ is known as the \emph{space of
wedges} of $X$. It is easy to deduce from Remark~\ref{universal_property_arcs} that this is canonically isomorphic to $J_{\infty}(J_{\infty}(X))$.
More generally, if we put $A_n=k\llbracket x_1,\ldots,x_n\rrbracket$, then we have a canonical isomorphism
$J_{\infty}(J_{A_n}(X))\simeq J_{A_{n+1}}(X)$ for every $n\geq 1$. 
\end{example}

\begin{remark}
It would be interesting to have explicit examples of schemes $J_A(X)$, when $X$ is singular. Very few such examples are known 
and all of these only deal with $J_m(X)$ or $J_{\infty}(X)$. Moreover, in almost all cases one can only describe the reduced structure on these spaces.
An easy example is that of schemes defined by monomial ideals in a polynomial ring (see \cite[Proposition~4.10]{Mus02}). A more interesting example
is that of $J_{\infty}(X)$, when $X$ is a toric variety. In this case, if $T$ is the torus acting on $X$, one can completely describe the orbits of the
$J_{\infty}(T)$-action on $J_{\infty}(X)$, see \cite{Ishii04}). It is much trickier to describe $J_m(X)$ for a toric variety $X$; this is only understood in
the $2$-dimensional case, see \cite{Mourtada11}. One example in which both
$J_m(X)$ and $J_{\infty}(X)$ are well understood is that of a generic determinantal variety $X$. In this case, $X$ is a closed subscheme
of a space of matrices $M=M_{m,n}(k)$ and the group $G=GL_m(k)\times GL_n(k)$ acts on $M$ inducing an action on $X$. 
For a description of the orbits of $J_m(G)$ on $J_m(X)$ and of the orbits of $J_{\infty}(G)$ on $J_{\infty}(X)$, see \cite{Docampo}.
\end{remark}

\section{Dimensions of jet schemes and invariants of singularities}

Beyond the formal properties that we discussed, little is known about jet schemes in the generality that we considered in the previous two sections. 
We now restrict to the case of the ``usual" jet schemes $J_m(X)$, for which a lot is known due to the connection to birational geometry
that comes out of the theory of motivic integration. 

In order to describe this connection, we first recall how one measures singularities in birational geometry. The idea is to use \emph{all} divisorial valuations,
suitably normalized by the order of vanishing along the relative Jacobian ideal. From now on we assume that the ground field $k$ is algebraically closed, of characteristic $0$,
and we only consider the $k$-valued points of the schemes involved. 

For simplicity, we assume that we work on an ambient smooth variety $X$. Let $\fra$ be a nonzero ideal on $X$ (all ideals are assumed to be coherent).
A \emph{divisor over} $X$ is a prime divisor on a normal variety $Y$ that has a birational morphism to $X$ ($Y$ is a \emph{birational model over} $X$). Each such divisor induces a valuation
${\rm ord}_E$ of the function field $K(Y)=K(X)$. 
We identify two such divisors if they give the same valuation. In particular, if $Z\to Y$ is a birational morphism, with $Z$ normal, then we identify $E$ with its strict transform on $Z$.
Therefore we may always assume that $Y$ is smooth (using a resolution of singularities of $Y$) and that $Y$ is proper over $X$ (using Nagata's compactification theorem).
Given a divisor $E$ over $X$, its center on $X$ denoted $c_X(E)$ is the closure of the image of $E$ in $X$ (it is easy to see that this is independent of the chosen model).

Let $E$ be a divisor over $X$. To a nonzero ideal $\fra$ on $X$, we can attach a nonnegative integer $\ord_E(\fra)$, defined as follows.
We may assume that $E$ is a divisor on a smooth variety $Y$ over $X$, such that the structural morphism $f\colon Y\to X$ factors through the blow-up
along $\fra$. Therefore we may write $\fra\cdot\cO_Y=\cO_Y(-D)$ for an effective divisor $D$ on $Y$ and $\ord_E(\fra)=\ord_E(D)$ is the coefficient of $E$ in $D$. 
For example, if $x\in X$ and $E$ is the exceptional divisor on the blow-up of $X$ at $x$, then $\ord_E(\fra)$ is the 
\emph{order of} $\fra$ at $x$, denoted by $\ord_x(\fra)$. This is characterized by
$$\ord_x(\fra):=\max\{r\mid\fra\subseteq\frm_x^r\},$$
where $\frm_x$ is the ideal defining $x$. 

The idea is to measure the singularities of $\fra$ using all invariants $\ord_E(\fra)$, where $E$ varies over the divisors over $X$. Very roughly, one
thinks of the singularities of $\fra$ being ``worse" if $\ord_E(\fra)$ is larger. On the other hand, when we vary $E$, the numbers $\ord_E(\fra)$ are unbounded,
hence we need a normalizing factor. It turns out that the right factor to use is provided by the \emph{log discrepancy}, which is defined as follows. If 
$f\colon Y\to X$ is a proper, birational morphism, with $Y$ a smooth variety and
$E$ is a prime divisor 
on $Y$, then the \emph{relative canonical class} $K_{Y/X}$ is the degeneracy locus of the morphism of
vector bundles of the same rank $f^*(\Omega_X)\to\Omega_Y$. In other words, $K_{Y/X}$ is locally defined by the determinant of the Jacobian
matrix of $f$. The log discrepancy of $\ord_E$ is $A(\ord_E):=\ord_E(K_{Y/X})+1$. It is easy to check that the definition does not depend on the model $Y$
we have chosen. 

There are various related invariants of singularities that one considers in birational geometry. In what follows, we focus on two such invariants,
the \emph{log canonical threshold} and the \emph{minimal log discrepancy}. We begin by introducing the former invariant.
If $X$ is smooth and $\fra$ is a nonzero ideal on $X$, then the log canonical threshold of $\fra$ is
$$\lct(\fra):=\inf_E\frac{A(\ord_E)}{\ord_E(\fra)},$$
where the infimum is over all divisors $E$ over $X$. 
Note that this is finite whenever $\fra\neq\cO_X$ and by convention we put $\lct(\cO_X)=\infty$ and $\lct(0)=0$.
If $W$ is the closed subscheme defined by $\fra$, we sometimes write $\lct(X,W)$ for $\lct(\fra)$.
Note that in the definition of $\lct(\fra)$ we consider the reciprocals of the invariants $\ord_E(\fra)$,
hence ``worse" singularities correspond to smaller log canonical thresholds.

While defined in terms of all divisors over $X$, it is a consequence of resolution of singularities
that the log canonical threshold can be computed on suitable models. Recall that a log resolution of the pair $(X,\fra)$ is a proper birational morphism
$f\colon Y\to X$, with $Y$ smooth, such that $\fra\cdot\cO_Y=\cO_Y(-D)$ for an effective divisor $D$ on $Y$ such that $D+K_{Y/X}$ has simple normal crossings\footnote{A divisor on a smooth variety has simple normal crossings
if around every point we can find local algebraic coordinates $x_1,\ldots,x_n$ such that the divisor is defined by $x_1^{a_1}\cdots x_n^{a_n}$ for some nonnegative integers $a_1,\ldots,a_n$.}.
The existence of such resolutions follows from Hironaka's fundamental results. A basic result about log canonical thresholds says that if $f\colon Y\to X$ is a log resolution of $(X,\fra)$ as above, then $\lct(\fra)$ is computed by the divisors on $Y$:
if we write $D=\sum_{i=1}^ra_iE_i$ and $K_{Y/X}=\sum_{i=1}^rk_iE_i$, then
$$\lct(\fra)=\min_{i=1}^r\frac{k_i+1}{a_i}.$$
A consequence of this formula is the fact, not apparent from the definition,  that $\lct(\fra)$ is a rational number. 

The log canonical threshold is a fundamental invariant of singularities. It appeared implicitly already in Atiyah's paper \cite{Atiyah} in connection with the meromorphic continuation of
complex powers. 
The first properties of the log canonical threshold have been proved by Varchenko
in connection with his work on
 asymptotic expansions of integrals and mixed Hodge structures on the vanishing cohomology, see \cite{Varchenko1}, \cite{Varchenko2},
 and \cite{Varchenko3}.
It was Shokurov who introduced the log canonical threshold in the context of birational geometry
in \cite{Sho}.  From this point of view, $\lct(\fra)$ is the largest rational number $q$ such that the pair $(X,\fra^q)$ is \emph{log canonical}.
We mention that the notion of
log canonical pairs is of central importance in the Minimal Model Program, since it gives 
the largest class of varieties for which one can hope to apply the program.
In fact, 
in the context of birational geometry it is useful to not require that the ambient variety is smooth,
but only that it has mild singularities, and it is in this more general setting that one can define the log canonical threshold. 
A remarkable feature of this invariant is that it is related to many points of view on singularities. We refer to \cite{Kol97} and \cite{Mus12}
for an overview of some of these connections and for the basic properties of the log canonical threshold.

\begin{example}
In order to illustrate the behavior of the log canonical threshold, we list a few examples.
When $\fra$ is generated by $f\in\cO(X)$, we simply write $\lct(f)$ for the corresponding log canonical threshold.
\begin{enumerate}
\item[i)] If the subscheme $V(\fra)$ defined by $\fra$ has codimension $r$, then $\lct(\fra)\leq r$. This is an equality if $V(\fra)$ is smooth.
\item[iii)] If $f=x_1^{a_1}\cdots x_n^{a_n}\in k[x_1,\ldots,x_n]$, then $\lct(f)=\min_i\frac{1}{a_i}$.
\item[iii)] If $f=x_1^{a_1}+\ldots+x_n^{a_n}\in k[x_1,\ldots,x_n]$, then $\lct(f)=\min\left\{1,\sum_i\frac{1}{a_i}\right\}$.
\end{enumerate}
\end{example}

\begin{remark}
In terms of size, the log canonical threshold in a neighborhood of a point $x\in X$ is comparable to $\ord_x(\fra)$. More precisely, 
given $x\in X$, there is an open neighborhood $U$ of $x$ such that the following inequalities hold:
$$\frac{1}{\ord_x(\fra)}\leq\lct(\fra\vert_U)\leq\frac{\dim(X)}{\ord_x(\fra)}.$$
\end{remark}

It turns out that the log canonical threshold governs the growth of the dimensions of the jet schemes $J_m(W)$ of a scheme $W$.
Note that by taking a finite affine open cover of $W$, we can reduce to the case when $W$ can be embedded in a smooth variety (for example,
in an affine space).

\begin{theorem}\label{connection_log_canonical_threshold}
If $X$ is a smooth $n$-dimensional variety and $W$
is a proper, nonempty, closed subscheme of $X$, then
$$\lim_{m\to\infty}\frac{\dim(J_m(W))}{m+1}=\max_m\frac{\dim(J_m(W))}{m+1}=n-\lct(X,W).$$
Moreover, the maximum above is achieved for all $m$ such that $(m+1)$ is divisible enough.
\end{theorem}

This result was proved in \cite{Mus02} by making use of the change of variable formula in motivic integration (in fact, the formula had appeared implicitly
earlier in \cite{DL98}). We explain below, following 
\cite{EinLazarsfeldMustata}, how this theorem follows from a more general result relating the approach to singularities via divisors over $X$ and that 
using certain subsets in the space of arcs $J_{\infty}(X)$. Before doing this, we discuss another invariant of singularities, whose definition is similar to that of the log canonical threshold,
but whose behavior turns out to be more difficult to study. 

Suppose, as above, that $X$ is a smooth variety, $\fra$ is a nonzero ideal on $X$, and $q$ is a positive rational number.
In order to avoid some pathologies of a trivial nature, we assume $\dim(X)\geq 2$. If $Z$ is a proper, irreducible closed subset of $X$,
then the \emph{minimal log discrepancy} 
$\mld_Z(X,\fra^q)$ is defined as
\begin{equation}\label{eq_def_mld0}
\mld_Z(X,\fra^q):=\inf\{A(\ord_E)-q\cdot \ord_E(\fra)\mid c_X(E)\subseteq Z\}.
\end{equation}
Note that ``better" singularities correspond to larger minimal log discrepancies.
It is a basic fact that the minimal log discrepancy is either $-\infty$ or it is nonnegative.
We have ${\rm mld}_{Z}(X,\fra^q)\geq 0$ if and only if
the pair $(X,\fra^q)$ is log canonical around $Z$, that is, 
there is an open neighborhood $U$ of $Z$ such that $\lct(\fra\vert_U)\geq q$.
Like the log canonical threshold, the minimal log discrepancy can be computed on suitable models. More precisely, if $I_Z$ is the ideal defining $Z$
and $f\colon Y\to X$ is a log resolution of $(X,I_Z\cdot\fra)$,
 then there is a prime  divisor $E$ on $Y$ that achieves the infimum in (\ref{eq_def_mld0}),
assuming that this infimum is not $-\infty$ (moreover, finitenes also can be tested just on the divisors on $Y$). 
For an introduction to minimal log discrepancies, we refer to \cite{ambro}.
The following result gives an interpretation of minimal log discrepancies in terms of jet schemes.

\begin{theorem}\label{connection_mld}
Let $X$ be a smooth variety of dimension $\geq 2$, $\fra$ a nonzero ideal on $X$ defining the subscheme $W$, and $q$ a positive rational number.
For every proper, irreducible closed subset $Z$ of $X$, we have 
$$\mld_Z(X,\fra^q)=\inf\{(m+1)(\dim(X)-q)-\dim(J_m(W)\cap (\pi_m^W)^{-1}(Z))\}.$$
Moreover, if the infimum is not $-\infty$, then it is a minimum.
\end{theorem}

The theorem was proved in \cite{EMY} using the change of variable formula in motivic integration. In this special form in which $X$ is assumed to be smooth,
it can also be deduced from the main result in \cite{EinLazarsfeldMustata} that we discuss below. 

We note that part of the interest in invariants of singularities like the log canonical threshold and the minimal log discrepancy comes from the fact that their behavior is related 
to one of the outstanding open problems in birational geometry, namely Termination of Flips (see \cite{Birkar} for the connection of this conjecture to log canonical thresholds
and \cite{Shokurov2} for the connection to minimal log discrepancies). In this respect, there are two points to keep in mind:

\noindent $\bullet$ For all applications to birational geometry, it is important to work with ambient varieties that have mild (log canonical) singularities, and not just smooth, as above.
In this context, one can still describe minimal log discrepancies in terms of properties of (certain subsets in) jet schemes, see \cite{EMY}. However, this description is less effective in general.
It was used to prove some open questions about minimal log discrepancies, such as Inversion of Adjunction and Semicontinuity in the case when the ambient variety
is locally complete intersection, see \cite{EMY} and \cite{EM04}. However, this method has not been successful so far in dealing with more general ambient varieties.

\noindent $\bullet$ While properties of log canonical thresholds are better understood (see below for the ACC property), it is in fact the (conjectural) properties of minimal log discrepancies that
would give a positive answer to the Termination of Flips Conjecture. More precisely, Shokurov has shown that two conjectures on minimal log discrepancies, the Semicontinuity Conjecture and the ACC Conjecture
imply Termination of Flips; see \cite{Shokurov2} for the precise statements. 

In order to give the flavor of ACC statements regarding invariants of singularities, we state the following result concerning log canonical thresholds.

\begin{theorem}\label{ACC_lct}
For every $n\geq 1$, the set of rational numbers 
$$\{\lct(\fra)\mid \fra\subsetneq\cO_X,\,X\,{\rm smooth},\,\dim(X)\leq n\}$$
satisfies ACC, that is, it contains no infinite strictly increasing sequences.
\end{theorem}

This result was conjectured by Shokurov and proved in \cite{dFEM10}. A general version, in which the ambient variety is not assumed to be smooth (which, as we have already mentioned,
is much more useful for birational geometry) has been recently proved in \cite{HMX}.
It is worth mentioning that a corresponding conjecture for minimal log discrepancies is widely open even in the case of smooth ambient varieties. For recent progress motivated by this question,
see  \cite{Kawakita1} and \cite{Kawakita2}.

Besides the ACC Conjecture for minimal log discrepancies, the other important open problem about these invariants concerns their semicontinuity
(this was conjectured by Ambro \cite{ambro}). As we have already mentioned, when the ambient variety is smooth, this can be deduced from
Theorem~\ref{connection_mld} using general properties of the dimension of algebraic varieties in families.

\begin{corollary}
If $X$ is a smooth variety, $\fra$ is a nonzero ideal on $X$, and $q$ is a positive rational number, then then the map
$$X\ni x\to \mld_x(X,\fra^q)$$
is lower semicontinuous. 
\end{corollary}

Our next goal is to describe more generally, following \cite{EinLazarsfeldMustata}, a dictionary between the approach to singularities using
divisorial valuations and that using certain subsets in the space of arcs. We fix a smooth variety $X$ of dimension $n$. We will be concerned with 
certain subsets in the space of arcs $J_{\infty}(X)$. In what follows we restrict to the $k$-valued points of $J_{\infty}(X)$, considered
as a topological space with the Zariski topology.
Recall that we have canonical projections $\pi_m\colon J_{\infty}(X)\to J_m(X)$. A \emph{cylinder} in $J_{\infty}(X)$ is a subset of the form $C=\pi_m^{-1}(S)$, 
where $S$ is a constructible subset in $J_m(X)$. We say that $C$ is closed, locally closed, or irreducible, if $S$ has this property. Moreover, we put
$$\codim(C):=\codim(S,J_m(X))=(m+1)n-\dim(S).$$
It is easy to see that all these notions do not depend on $m$, since the natural projections $J_{m+1}(X)\to J_m(X)$ are locally trivial with fiber $\AA^n$. 

The main examples arise as follows. Suppose that $\fra$ is a nonzero ideal on $X$, defining the subscheme $W$. Associated to $W$ we have a function
$\ord_W\colon J_{\infty}(X)\to \ZZ_{\geq 0}\cup\{\infty\}$ such that for $\gamma\colon\Spec k\llbracket t\rrbracket\to X$, we have
$$\ord_W(\gamma)=\ord_t(\gamma^{-1}(\fra))$$
(with the convention that this is $\infty$ if the ideal $\gamma^{-1}(\fra)$ is $0$).
With this notation, the \emph{contact locus}
$${\rm Cont}^{\geq m}(W):=\ord_W^{-1}(\geq m)$$
is a closed cylinder, hence
$${\rm Cont}^m(W):=\ord_W^{-1}(m)={\rm Cont}^{\geq m}(W)\smallsetminus {\rm Cont}^{\geq (m+1)}(W)$$
is a locally closed cylinder. In fact, we have
$${\rm Cont}^{\geq (m+1)}(W)=\pi_m^{-1}(J_m(W)).$$
In particular, we have 
$$\codim({\rm Cont}^{\geq (m+1)}(W))=(m+1)n-\dim(J_m(W)).$$

The main point of the correspondence we are going to describe is that divisorial valuations correspond to cylinders in such a way that the log discrepancy function
translates to the codimension of the cylinder. In order to simplify the exposition, let us assume that $X=\Spec(R)$ is affine.
Note first that if $C$ is an irreducible closed cylinder in $J_{\infty}(X)$, then we may define a map 
$$\ord_C\colon R\to \ZZ_{\geq 0}\cup \{\infty\},\,\,\ord_C(f):=\min\{\ord_{V(f)}(\gamma)\mid \gamma\in C\}.$$
This satisfies 
$$\ord_C(f+g)\geq\min\{\ord_C(f),\ord_C(g)\}\,\,\text{and}\,\,\ord_C(fg)=\ord_C(f)+\ord_C(g).$$
Moreover, if $C$ does not dominate $X$, then $\ord_C(f)<\infty$ for every nonzero $f$, hence $\ord_C$ extends to a valuation of the function field of $X$.
The following is the main result from \cite{EinLazarsfeldMustata} concerning the description of divisorial valuations in terms of cylinders in the space or arcs.
For an extension to singular varieties, see \cite{dFEI}.

\begin{theorem}\label{ELM}
Let $X$ be a smooth variety.
\begin{enumerate}
\item[i)] If $C$ is an irreducible closed cylinder in $J_{\infty}(X)$ that does not dominate $X$, then
there is a divisor $E$ over $X$ and a positive integer $q$ such that $\ord_C=q\cdot\ord_E$.
\item[ii)] For every divisor $E$ over $X$ and every positive integer $q$, there is a closed irreducible
cylinder $C$ in $J_{\infty}(X)$ such that $\ord_C=q\cdot\ord_E$. Moreover, there is a unique maximal such
cylinder $C=C_q(E)$, with respect to inclusion, and
$$\codim(C_{q}(E))=q\cdot A(\ord_E).$$
\end{enumerate}
\end{theorem}

It is easy to see that this result implies both Theorems~\ref{connection_log_canonical_threshold}
and \ref{connection_mld}. The key step in the proof of Theorem~\ref{ELM} consists in analyzing the
valuations corresponding to the irreducible  components of ${\rm Cont}^{\geq m}(W)$, for a closed subscheme $W$. 
This is done by considering a log resolution $f\colon Y\to X$ of $(X,\fra)$, where $\fra$ is the ideal defining $W$.
In this case, we have an induced map $g=J_{\infty}(f)\colon J_{\infty}(Y)\to J_{\infty}(X)$. If $f$ is an isomorphism over
$X\smallsetminus A$, where $A$ is a proper closed subset of $X$, then the valuative criterion for properness implies 
that $g$ is a bijection over $J_{\infty}(Y)\smallsetminus J_{\infty}(A)$. While $J_{\infty}(A)$ is ``small" in $J_{\infty}(X)$
(for example, if the ground field is uncountable, then no cylinder is contained in $J_{\infty}(A)$), the map $g$
is far from being a homeomorphism over $J_{\infty}(Y)\smallsetminus J_{\infty}(A)$. In fact, it is a fundamental result that if $C$ is a cylinder contained in ${\rm Cont}^e(K_{Y/X})$, then $g(C)$ is again a cylinder and
$$\codim(g(C))=\codim(C)+e.$$
This is a consequence of the geometric statement behind the change of variable formula in motivic integration,
due to \cite{Kontsevich}. For the relevant statement and its proof, as well as for an important generalization to the case when $X$ is not smooth,
we refer to \cite{DL99}.
Since ${\rm Cont}^{\geq m}(f^{-1}(W))=g^{-1}({\rm Cont}^{\geq m}(W))$, one can deduce the following formula for the codimension of
${\rm Cont}^{\geq m}(W)$, see \cite{EinLazarsfeldMustata} for details.

\begin{theorem}\label{formula_contact_loci}
Let $X$ be a smooth variety and $W$ a proper closed subscheme of $X$, defined by the ideal $\fra$.
If $f\colon Y\to X$ is a log resolution of the pair $(X,\fra)$ and we write
$$f^{-1}(W)=\sum_{i=1}^ra_iE_i\,\,\text{and}\,\,K_{Y/X}=\sum_{i=1}^rk_iE_i,$$
then $\codim({\rm Cont}^{\geq m}(W))$ is equal to
$$\min\left\{\sum_{i=1}^r(k_i+1)\nu_i\mid \nu=(\nu_1,\ldots,\nu_r)\in\ZZ_{\geq 0}^r,\bigcap_{\nu_i>0}E_i\neq\emptyset,\sum_{i=1}^ra_i\nu_i\geq m\right\}.$$
\end{theorem}

Moreover, it is easy to see that in the setting of this theorem, every irreducible component $C$ of ${\rm Cont}^{\geq m}(W)$ is the closure of the image of a multi-contact locus of the form
$\cap_{i=1}^r{\rm Cont}^{\geq\nu_i}(E_i)$. Using this, one deduces that $\ord_C$ is equal to $q\cdot\ord_E$, where $E$ is the exceptional divisor on a suitable weighted blow-up of $Y$
with respect to the simple normal crossing divisor $\sum_iE_i$. We note that since both log canonical thresholds and minimal log discrepancies can be computed using log resolutions, the statement
of Theorem~\ref{formula_contact_loci} is enough to imply Theorems~\ref{connection_log_canonical_threshold}
and \ref{connection_mld}. On the other hand, one can give a proof for Theorem~\ref{ELM} without using log resolutions, see \cite{Zhu}. An advantage of that approach is that one obtains the
same result in positive characteristic. 

\begin{remark}\label{rmk_monotonicity}
It is an immediate consequence of Theorem~\ref{formula_contact_loci} that if $W$ is any scheme, then
$$\frac{\dim(J_{m-1}(W))}{m}\leq\frac{\dim(J_{mp-1}(W))}{mp}$$
for every positive integers $m$ and $p$. It would be interesting to give a direct proof of this inequality without relying on log resolutions.
Such an argument could then hopefully be extended to cover other jet schemes $J_A(X)$,
for suitable $A$.
\end{remark}

\begin{remark}
It was shown in \cite{Mus02} that one can use the description of the log canonical threshold in Theorem~\ref{connection_log_canonical_threshold}
to reprove some of the basic properties of this invariant. For example, one can use this approach to prove the following special case
of Inversion of Adjunction: if $X$ is a smooth variety and $H$ is a smooth hypersurface in $X$, then for every ideal $\fra$ on $X$ such that
$\fra\cdot\cO_H\neq 0$, there is an open neighborhood $U$ of $H$ such that
\begin{equation}\label{eq_inversion_adjunction}
\lct(\fra\vert_U)\geq\lct(\fra\cdot\cO_H).
\end{equation}
The usual proof of this inequality makes use of vanishing theorems (see for example \cite{Kol97}). 
Since Theorem~\ref{formula_contact_loci} also holds in positive characteristic, one deduces that 
the inequality (\ref{eq_inversion_adjunction}) holds in this setting as well, see \cite{Zhu}, in spite of the fact that
vanishing theorems can fail.
\end{remark}

\begin{remark}
It would be interesting to find a jet-theoretic proof of the following result of Varchenko. Suppose that $T$ is a connected scheme and
${\mathcal W}\hookrightarrow \AA^n\times T$ is an effective Cartier divisor, flat over $T$, such that for every (closed) point $t\in T$, the induced divisor
${\mathcal W}_t\hookrightarrow\AA^n$ has an isolated singularity at $0$. By using the connection between the log canonical threshold and Steenbrink's 
spectrum of a hypersurface singularity, Varchenko showed in \cite{Varchenko1} that if the Milnor number at $0$ for each
${\mathcal W}_t$ is constant for $t\in T$, then
$\lct_0(\AA^n,{\mathcal W}_t)$ is independent of $t$
(here $\lct_0(\AA^n,{\mathcal W}_t)=\lct(U,{\mathcal W}_t\cap U)$, where $U$ is a small neighborhood of $0$).
It would be interesting to deduce this fact from the behavior of jet schemes. 
It is easy to see that the jet schemes $J_m({\mathcal W}_t)$ are the fibers of a relative jet scheme $J_m({\mathcal W}/T)$ over $T$ and
a natural question is whether the constancy of Milnor numbers implies
that this family is flat over $T$ (in a neighborhood of the fiber over $0$ via the natural map $J_m({\mathcal W}/T)\to\AA^n$).
\end{remark}

\begin{remark}
Suppose, for simplicity, that $X=\Spec(R)$ is a smooth affine variety. Recall that if ${\mathfrak a}$ is an ideal in $R$, then its
integral closure $\overline{{\mathfrak a}}$ consists of all $\phi\in R$ such that $\ord_E(\phi)\geq\ord_E({\mathfrak a})$ for all divisors
$E$ over $X$. It follows from Theorem~\ref{ELM} that if ${\mathfrak b}$ is another ideal in $R$, then 
${\mathfrak b}\subseteq \overline{{\mathfrak a}}$ if and only if ${\rm Cont}^{\geq m}({\mathfrak a})\subseteq {\rm Cont}^{\geq m}({\mathfrak b})$
for all $m$. In particular, the integral closure $\overline{{\mathfrak a}}$ is determined by the contact loci $({\rm Cont}^{\geq m}({\mathfrak a}))_{m\geq 1}$.
Given other invariants of an ideal that only depend on the integral closure, it would be interesting to find a direct description of these invariants in terms of the
contact loci of that ideal. For example, if ${\mathfrak a}$ is supported at a point $x\in X$, it would be interesting to find a description of the
Samuel multiplicity $e({\mathfrak a}\cdot {\mathcal O}_{X,x}, {\mathcal O}_{X,x})$ in terms of the contact loci of ${\mathfrak a}$. 
\end{remark} 

We now turn to a related connection between singularities and jet schemes. It turns out that in the case of 
locally complete intersection varieties, one can characterize various classes of singularities in terms of the behavior of the jet schemes.
The following is the main result in this direction.

\begin{theorem}\label{irred_jet_schemes}
Let $W$ be a locally complete intersection variety.
\begin{enumerate}
\item[i)] If $W$ is normal, then $W$ has log canonical singularities if and only if $J_m(W)$ has pure dimension for every $m\geq 0$.
Moreover, in this case we have $\dim(J_m(W))=(m+1)\cdot\dim(X)$ and $J_m(W)$ is a locally complete intersection.
\item[ii)] The variety $W$ has rational (equivalently, canonical) singularities if and only if $J_m(W)$ is irreducible for every $m\geq 0$.
\item[iii)] The variety $W$ has terminal singularities if and only if $J_m(W)$ is normal for every $m\geq 0$. 
\end{enumerate}
\end{theorem}

The assertion in ii) was proved in \cite{Mus01} using motivic integration. It was then noticed in \cite{EinLazarsfeldMustata} that the main ingredient
in the proof can also be deduced from Theorem~\ref{ELM}. The result had been conjectured by David Eisenbud and Edward Frenkel. They use it in
the appendix to \cite{Mus01} to give an analogue of a result of Kostant to the setting of loop Lie algebras. The assertions in i) and iii) follow using similar ideas
once some basic Inversion of Adjunction statements are proved, see \cite{EMY} and \cite{EM04}. 

We end with an interpretation for the condition of having pure-dimensional or irreducible jet schemes under the assumptions of 
Theorem~\ref{irred_jet_schemes}. It turns out that these considerations can be made in the context of generalized jet schemes and we will
make use of this in the next section.

\begin{proposition}\label{description_equidimensional}
Let $W$ be an $n$-dimensional locally complete intersection variety and let $A$ be a local, finite $k$-algebra, with $\dim_k(A)=\ell$.
\begin{enumerate}
\item[i)] $J_A(W)$ is pure-dimensional if and only if all irreducible components of $J_A(W)$ have dimension $\ell n$.
If this is the case, then $J_A(W)$ is locally a complete intersection.
\item[ii)] $J_A(W)$ is irreducible if and only if the inverse image in $J_A(W)$ of the singular locus of $W$ has 
dimension $<\ell n$. If this is the case, then $J_A(W)$ is also reduced.
\item[iii)] Suppose that $A\to A'$ is a surjective $k$-algebra homomorphism, where $A$ and $A'$ are finite, local $k$-algebras,
with $A'$ being positively graded. If $J_{A}(W)$ is pure-dimensional or irreducible, then $J_{A'}(W)$ has the same property.
\end{enumerate}
\end{proposition}

\begin{proof}
If we write $W=U_1\cup\ldots\cup U_m$, with each $U_i$ open in $X$, then $J_A(W)=J_A(U_1)\cup\ldots\cup J_A(U_m)$. Using this, it is easy to see
that if the assertions in the proposition hold for each $U_i$, then they hold for $X$. Therefore we may assume 
$W$ is a closed subvariety of $X=\AA^d$, whose ideal is generated by $r=d-n$ elements.

We have seen in Proposition~\ref{prop_existence_jet_scheme} that $J_A(W)$ is cut out in $J_A(X)\simeq\AA^{\ell d}$
by $\ell r$ equations. Therefore each irreducible component of $J_A(X)$ has dimension $\geq \ell n$. On the other hand, it follows from 
Corollary~\ref{smooth_case} that $J_A(X_{\rm sm})$ is an open subset of $J_A(X)$ of dimension $\ell n$. We thus conclude that $J_A(X)$ is pure-dimensional
if and only if all irreducible components of $J_A(X)$ have dimension $\ell n$. We also see that if this is the case, then $J_A(X)$ is itself a locally complete intersection. 
This proves i).

Note also that  $J_A(W)$ is irreducible if and only if the inverse image of the singular locus $W_{\rm sing}$ in $J_A(W)$ has dimension $<\ell n$. If this is the case,
then $J_A(W)$ is generically reduced and being Cohen-Macaulay (recall that $J_A(W)$ has to be a locally complete intersection), it is reduced. This gives ii).

In order to prove iii), let $\ell'=\dim_k(A')$. It is enough to show that if $Z$ is any closed subset of $W$, then 
\begin{equation}\label{eq1_description_equidimensional}
\dim\,(\pi_{A'}^W)^{-1}(Z)\leq \dim\,(\pi_{A}^W)^{-1}(Z)-(\ell-\ell')n.
\end{equation}
It follows from Proposition~\ref{prop_existence_jet_scheme} that 
\begin{equation}\label{eq2}
J_A(W)\hookrightarrow (\pi_{A'/A}^{\AA^d})^{-1}(J_{A'}(W))
\end{equation}
is cut out by $r(\ell-\ell')$ equations and the same holds if we restrict to the inverse images of $Z$.
By putting these together, we obtain the inequality (\ref{eq1_description_equidimensional}) if we can show that for every
irreducible component $R$ of $(\pi_{A'}^W)^{-1}(Z)$, its inverse image $(\pi_{A'/A}^{\AA^d})^{-1}(R)$ intersects $(\pi_A^W)^{-1}(Z)$.
This is a consequence of the fact that if $x\in Z$, then $s_{A}^X(x)$ lies in this intersection by Remark~\ref{action}.
\end{proof}

\section{Some questions on generalized jet schemes}

In this section we collect some questions and remarks concerning the behavior of the schemes $J_A(X)$, when ${\rm embdim}(A)\geq 2$.
Very little is known in this context, partly due to a lack of examples. In what follows we work over an algebraically closed field
$k$ of characteristic zero.

Motivated by Theorem~\ref{connection_log_canonical_threshold}, we begin by proposing several invariants that measure the rate of growth 
of the dimensions of certain schemes $J_A(X)$. In order to simplify the notation, we restrict to the first unknown case, that when ${\rm embdim}(A)=2$.
We introduce three invariants, depending on the choice of algebras $A$.

We first consider the algebras $A_{p,q}=k[s,t]/(s^p,t^q)$, with $p,q\geq 1$. Note that $\dim_k(A)=pq$. Given a scheme $X$, let
$$\alpha_{p,q}(X):=\dim J_{A_{p,q}}(X)$$ and
$$\alpha(X):=\sup_{p,q\geq 1}\frac{\alpha_{p,q}(X)}{pq}.$$
Note that $\alpha(X)\leq\max_{x\in X}\dim(T_xX)$ by Corollary~\ref{bound_dimension_jet_schemes}.
Since $J_{A_{p,q}}(X)\simeq J_{p-1}(J_{q-1}(X))$, it follows from Remark~\ref{rmk_monotonicity} that for every positive integer $m$, we have
\begin{equation}\label{eq1_gen_invariants}
\frac{\alpha_{p,q}(X)}{pq}\leq \frac{\alpha_{mp,q}(X)}{mpq}\,\,\text{and}\,\,\frac{\alpha_{p,q}(X)}{pq}\leq\frac{\alpha_{p,mq}(X)}{mpq}.
\end{equation}
This clearly implies
\begin{equation}\label{eq3_gen_invariants}
\alpha(X)=\sup_{p\geq 1}\frac{\alpha_{p,p}(X)}{p^2}=\limsup_{p\to\infty}\frac{\alpha_{p,p}(X)}{p^2}.
\end{equation}
On the other hand, it follows from Theorem~\ref{connection_log_canonical_threshold} that if $X$ is a closed subscheme of the smooth variety $Y$, then for every $q\geq 1$, we have
\begin{equation}\label{eq2_gen_invariants}
\lim_{p\to\infty}\frac{\alpha_{p,q}(X)}{pq}=\sup_{p\geq 1}\frac{\alpha_{p,q}(X)}{pq}=\dim(Y)-\frac{\lct(J_{q-1}(Y),J_{q-1}(X))}{q}.
\end{equation}
It is easy to deduce from (\ref{eq1_gen_invariants}) and (\ref{eq2_gen_invariants}) the following proposition.

\begin{proposition}
If $X$ is a closed subscheme of the smooth variety $Y$, then
$$\alpha(X)=\dim(Y)-\inf_{q\geq 1}\frac{\lct(J_{q-1}(Y),J_{q-1}(X))}{q}=\dim(Y)-\liminf_{q\to\infty}\frac{\lct(J_{q-1}(Y),J_{q-1}(X))}{q}.$$
\end{proposition}

We now turn to another invariant, corresponding to a different sequence of algebras. For every $m\geq 1$, let $A_m=k[s,t]/(s,t)^m$.
Note that $\dim_k(A_m)=\frac{m(m+1)}{2}$. For a scheme $X$ over $k$, we put $\beta_m:=\dim(J_{A_m}(X))$ and
\begin{equation}\label{eq4_gen_invariants}
\beta(X):=\sup_{m\geq 1}\frac{\beta_m(X)}{m(m+1)/2}.
\end{equation}
It follows from Corollary~\ref{bound_dimension_jet_schemes} that $\beta(X)\leq \max_{x\in X}\dim(T_xX)$.

\begin{question}
Is the supremum in (\ref{eq4_gen_invariants}) also the limsup of the corresponding sequence? Is there any relation between $\alpha(X)$ and $\beta(X)$?
\end{question}

\begin{example}
The invariant $\beta(X)$ is slightly easier to compute than $\alpha(X)$. For example, suppose that $X$ is defined in ${\mathbf A}^n=\Spec k[x_1,\ldots,x_n]$
by $x_1^{a_1}\ldots x_n^{a_n}$, for nonnegative integers $a_1,\ldots,a_n$, not all of them equal to $0$. An element of $J_{A_m}(X)$ corresponds to
a $k$-algebra homomorphism $\phi\colon k[x_1,\ldots,x_n]\to k[s,t]/(s,t)^m$, which is determined by $\phi(x_1),\ldots,\phi(x_n)$, such that $\prod_i\phi(x_i)^{a_i}=0$.
If we denote by $\nu_i$ the smallest power of $(s,t)$ that contains $\phi(x_i)$, we see that we get 
a disjoint decomposition of $J_{A_m}(X)$ into locally closed subsets $J_{A_m}(X)_{\nu}$, parametrized by $\nu=(\nu_1,\ldots,\nu_n)
\in\{0,1,\ldots,m\}^n$ such that $\sum_{i=1}^na_i\nu_i\geq m$. It is straightforward to see that
$$\dim(J_{A_m}(X)_{\nu})=\frac{nm(m+1)}{2}-\sum_{i=1}^n\frac{\nu_i(\nu_i+1)}{2}$$
and an easy computation shows that since $\sum_{i=1}^na_i\nu_i\geq m$, we have
\begin{equation}\label{eq5_gen_invariants}
\sum_{i=1}^n\nu_i(\nu_i+1)-\frac{m(m+1)}{\sum_{i=1}^na_i^2}\geq 0.
\end{equation}
Moreover, if we take $\nu_i=a_i\ell$ for some $\ell\geq 1$ and $m=\ell\cdot\sum_ia_i^2$, then
the left-hand side of (\ref{eq5_gen_invariants}) is equal to $\ell\left(\sum_{i=1}^na_i-1\right)$.
This implies that
$$\beta(X)=\limsup_{m\to\infty}\frac{\beta_m(X)}{m(m+1)/2}=n-\frac{1}{\sum_{i=1}^na_i^2}.$$
\end{example}

Yet another invariant of a similar flavor is the following. If $X$ is a scheme over $k$, then let
$$\gamma(X):=\sup_A\frac{\dim(J_A(X))}{\dim_k(A)},$$
where the supremum is over all algebras $A\in {\rm LFA}/k$ which are graded\footnote{Of course, it might make sense to remove the condition that $A$ is graded. We do not know whether 
this would give a different invariant.} and such that ${\rm embdim}(A)\leq 2$. It is clear from definition that $\alpha(X),\beta(X)\leq\gamma(X)$, while
Corollary~\ref{bound_dimension_jet_schemes} implies that $\gamma(X)\leq \max_{x\in X}\dim(T_xX)$.

We now turn to a different question, motivated by Theorem~\ref{irred_jet_schemes} and inspired by \cite{ISW12}. 
Suppose that $X$ is a locally complete intersection variety and $r$ is a positive integer. We say that $X$ has irreducible (resp. pure-dimensional) $r$-iterated jet schemes
if the jet scheme $J_{m_1}(J_{m_2}(\ldots J_{m_r}(X)))$ is irreducible (resp. pure-dimensional) for all nonnegative integers $m_1,\ldots,m_r$.

\begin{proposition}\label{prop_IJS}
For a locally complete intersection variety $X$, the following are equivalent:
\begin{enumerate}
\item[i)] $X$ has irreducible (resp. pure-dimensional) $r$-iterated jet schemes.
\item[ii)] $J_A(X)$ is irreducible (resp. pure-dimensional) for every $A\in {\rm LFA}/k$ which is graded and such that
 ${\rm embdim}(A)\leq r$.
\item[iii)] $J_{A_m}(X)$ is irreducible (resp. pure-dimensional) for every $m\geq 1$, where
$A_m=k[x_1,\ldots,x_r]/(x_1,\ldots,x_r)^m$.
\end{enumerate}
Furthermore, $X$ has irreducible $r$-iterated jet schemes if and only if all $(r-1)$-iterated jet schemes of $X$ have rational singularities. 
If this is the case, then all $r$-iterated jet schemes of $X$ are reduced locally complete intersections.
\end{proposition}

The equivalence of i)-iii) is an easy consequence of Proposition~\ref{description_equidimensional},
which together with Theorem~\ref{irred_jet_schemes} also gives the last assertions in the proposition.
When $r\geq 2$, it is not easy to give examples of singular varieties which have all $r$-iterated jet schemes
pure-dimensional. We discuss below the case of cones over smooth projective hypersurfaces, which provides
the only nontrivial class of examples for $r=2$.

\begin{remark}
Note that if $X$ is a locally complete intersection variety, then it follows from Proposition~\ref{prop_IJS} that $X$ has pure-dimensional 2-iterated jet schemes
if and only if $\alpha(X)=\dim(X)$, which is also equivalent to saying that either $\beta(X)=\dim(X)$ or that $\gamma(X)=\dim(X)$. 
\end{remark}

\begin{remark}
We note that Example~\ref{example1} implies that if $x\in X$ is a singular point such that
$$r>\frac{\dim(X)}{\dim(T_xX)-\dim(X)},$$
then $X$ does not have pure-dimensional $r$-iterated jet schemes.
\end{remark}

Suppose now that $X\subseteq \AA^n$, with $n\geq 3$, is defined by a homogeneous polynomial $f$ of degree $d>0$. We assume that $X\smallsetminus \{0\}$
is smooth, which is the case for general $f$. It is well-known that $X$ has rational (log canonical) singularities if and only if $d<n$ (resp., $d\leq n$). Using Theorem~\ref{irred_jet_schemes},
we conclude that $X$ has irreducible (pure-dimensional) 1-iterated jet schemes if and only if $d<n$ (resp., $d\leq n$). 
We give a direct argument for this, in the spirit of \cite[\S 3]{dFEM03}.

\begin{example}\label{jet_scheme_homogeneous_hypersurface}
We show that if $X=V(f)\subset \AA^n$, with $f$ homogeneous of degree $d>0$, such that $X$ has an isolated singularity at $0$,
then $J_m(X)$ is irreducible (pure-dimensional) for all $m\geq 1$ if and only if $d<n$ (resp., $d\leq n$).
Let us denote by $\pi_m\colon J_m(X)\to X$ and  $\pi'_m\colon J_m(\AA^n)\to\AA^n$ the canonical projections.
It follows from Proposition~\ref{description_equidimensional} that we need to show the following:
we have $\dim(\pi_m^{-1}(0))<(m+1)(n-1)$ (resp., 
$\dim(\pi_m^{-1}(0))\leq (m+1)(n-1)$) for all $m\geq 1$
if and only if $d<n$ (resp., $d\leq n$). 
We use the following two facts:
\begin{enumerate}
\item[i)] If $m\leq (d-1)$, then the closed embedding $J_m(X)\hookrightarrow
J_m(\AA^n)$ induces an isomorphism $\pi_m^{-1}(0)\simeq{\pi'_m}^{-1}(0)\simeq\AA^{mn}$.
\item[ii)] If $m\geq d$, then we have an isomorphism $\pi_m^{-1}(0)\simeq J_{m-d}(X)
\times\AA^{n(d-1)}$.
\end{enumerate}
Both assertions follow from the universal property defining $J_m(X)$, together with the following observations: if $R$ is a $k$-algebra and $u_1,\ldots,u_n\in tR[t]/(t^{m+1})$, then
$f(u_1,\ldots,u_n)=0$ whenever $m\leq d-1$; for $m\geq d$, we have
$f(u_1,\ldots,u_n)=0$ if and only if when we write $u_i=tv_i$, we have
$f(\overline{v_1},\ldots,\overline{v_n})=0$ in $R[t]/(t^{m+1-d})$, where
$\overline{v_i}$ is the class of $v_i$ in $R[t]/(t^{m+1-d})$.

In particular, it follows from i) that $\dim(\pi_{d-1}^{-1}(0))=(d-1)n$. If
$\dim(\pi_{d-1}^{-1}(0))<d(n-1)$, we deduce that $d<n$.
Conversely, suppose that $d<n$. We prove by induction on $m$ that
$\dim(\pi_m^{-1}(0))<(m+1)(n-1)$. If $0\leq m\leq d-1$, this follows easily from i).
On the other hand,  if $m\geq d$, then the assertion follows from ii) and the inductive hypothesis. One similarly shows that $d\leq n$ if and only if $\dim(\pi_m^{-1}(0))\leq (m+1)(n-1)$ for all $m$.
\end{example}

With the above notation, we are interested in when the $r$-iterated jet schemes of $X$ are irreducible or pure-dimensional for $r\geq 2$. 
It is easy to give a necessary condition, arguing as in Example~\ref{jet_scheme_homogeneous_hypersurface}.

\begin{proposition}\label{prop_jets_homog_hyp}
Let $X\subset\AA^n$  be defined by a homogeneous polynomial $f$ of degree $d>0$. If $r\geq 2$ and the $r$-iterated jet schemes of 
$X$ are pure-dimensional, then $d^r\leq n$.
\end{proposition}

\begin{proof}
For every positive integer $j$, we consider $A_j=k[t_1,\ldots,t_r]/(t_1,\ldots,t_r)^{jd}$.
Note that for every $k$-algebra $R$ and every $u_1,\ldots,u_n\in 
(t_1,\ldots,t_r)^{j}/(t_1,\ldots,t_r)^{jd}\subset R\otimes_k A_j$, we have
$f(u_1,\ldots,u_n)=0$. This shows that $J_{A_j}(X)$ contains a closed subscheme $Z_j$,
with
$$\dim(Z_j)=n\cdot \dim_k (t_1,\ldots,t_r)^jA_j=n\left({{jd+r-1}\choose {r}}-{{j+r-1}\choose {r}}\right).$$
Using Proposition~\ref{description_equidimensional} and the fact that $\dim_k(A_j)={{jd+r-1}\choose{r}}$,
we deduce from our assumption that
$$
n\left({{jd+r-1}\choose {r}}-{{j+r-1}\choose {r}}\right)\leq(n-1)\cdot {{jd+r-1}\choose{r}}.
$$
This gives
\begin{equation}\label{eq1_proof}
{{jd+r-1}\choose{r}}\leq n\cdot {{j+r-1}\choose {r}},
\end{equation}
which we can rewrite as 
\begin{equation}\label{eq2_proof}
n\geq d\cdot \prod_{i=1}^{r-1}\frac{jd+i}{j+i}.
\end{equation}
Since each function $\phi_i(x)=\frac{dx+i}{x+i}$, with $1\leq i\leq r-1$, is increasing, with 
$\lim_{x\to\infty}\phi_i(x)=d$, we conclude from (\ref{eq2_proof}) by letting $j$ go to infinity
that $n\geq d^r$.
\end{proof}

\begin{remark}
If in the above proposition we assume instead that the $r$-iterated jet schemes of $X$ are irreducible, then in 
(\ref{eq2_proof}) we get strict inequality. However, this does not imply that $n>d^r$. 
\end{remark}

\begin{question}\label{last_question}
Does the converse to the assertion in Proposition~\ref{prop_jets_homog_hyp} hold when $X$ has an isolated singularity?
More precisely, suppose that $X$ is defined in $\AA^n$ by a homogeneous polynomial $f$ of degree $d$, such that $X\smallsetminus\{0\}$ is smooth.
If $r\geq 2$ is such that $n\geq d^r$, are all $r$-iterated jet schemes of $X$ pure-dimensional? Are they irreducible? Do these assertions 
hold if $f$ is general?
\end{question}

\begin{remark}
The only evidence for a positive answer to Question~\ref{last_question} is provided by the result from \cite{ISW12}, saying that 
if $f=\sum_ua_ux^u\in k[x_1,\ldots,x_n]$, where the sum is over all $u=(u_1,\ldots,u_n)\in{\mathbf Z}_{\geq 0}^n$ with $\sum_iu_i=d$,
then if $d^2\leq n$ and the coefficients $(a_u)_u$ are algebraically independent over ${\mathbf Q}$, then the $2$-iterated jet schemes
of $X=V(f)$ are irreducible. The proof shows that all $J_m(X)$ have rational singularities by reducing to positive characteristic and using 
the theory of $F$-singularities. In particular, this shows that if the ground field $k$ is uncountable, then for a very general polynomial, the
$2$-iterated jet schemes of $V(f)$ are irreducible. 
\end{remark}

\begin{remark}
It is interesting that the bound in Question~\ref{last_question} is (almost) the same as the bound that shows up in Lang's $C_r$ condition on fields. Recall that if
$r\geq 0$, then a field $K$ satisfies the condition $C_r$ if every homogeneous polynomial $f\in K[x_1,\ldots,x_n]$ of degree $d$, with $d^r<n$,
has a nontrivial zero in $K^n$. For example, if $k$ is algebraically closed, then it is known that the field $K=k(x_1,\ldots,x_r)$ satisfies condition $C_r$. 
It would be interesting if there was a connection between $C_r$ fields and Question~\ref{last_question}.
\end{remark}

\subsection*{Acknowledgments}
I am indebted to David Eisenbud and Edward Frenkel who introduced me to jet schemes and shared with me their conjecture on the
irreducibility of jet schemes of locally complete intersection varieties. I am also grateful to Lawrence Ein, from whom I learned a great deal over the years,
during our collaboration. In particular, several of the results discussed in this article have been obtained in joint work with him.


\end{document}